\documentclass[a4paper,11pt,table]{article}

\usepackage[T1]{fontenc}
\usepackage[latin9]{inputenc}
\usepackage{geometry}
\usepackage{amsmath}
\usepackage{amsthm}
\usepackage{amssymb}
\usepackage{bbm}
\usepackage{color}
\usepackage{paralist}
\usepackage{todonotes}
\usepackage{mathtools}
\usepackage[unicode=true,pdfusetitle,
 bookmarks=true,bookmarksnumbered=false,bookmarksopen=false,
 breaklinks=false,pdfborder={0 0 0},pdfborderstyle={},backref=false,colorlinks=false]
 {hyperref}
\usepackage[labelfont=bf,labelsep=period]{caption}
\numberwithin{figure}{section}

\makeatletter

\usepackage[nameinlink,capitalise,noabbrev]{cleveref}

\hypersetup{%
    bookmarksnumbered, bookmarksopen=true, bookmarksopenlevel=1,%
}

\theoremstyle{plain}
\newtheorem{thm}{Theorem}[section]
\crefname{thm}{Theorem}{Theorems}
\theoremstyle{plain}
\newtheorem{lem}[thm]{Lemma}
\crefname{lem}{Lemma}{Lemmas}
\theoremstyle{plain}

\theoremstyle{plain}
\newtheorem*{claim*}{Claim}
\crefname{claim}{Claim}{Claims}
\theoremstyle{definition}
\newtheorem{defn}[thm]{Definition}
\newtheorem{problem}[thm]{Problem}
\theoremstyle{plain}
\newtheorem{conjecture}[thm]{Conjecture}
\crefname{conjecture}{Conjecture}{Conjectures}
\theoremstyle{plain}

\crefname{prop}{Proposition}{Propositions}
\theoremstyle{definition}

\theoremstyle{definition}
\newtheorem*{rem*}{Remark}
\theoremstyle{plain}

\theoremstyle{plain}

\theoremstyle{observation}

\crefformat{equation}{#2(#1)#3}

\date{}

\crefformat{equation}{#2(#1)#3}

\let\originalleft\left
\let\originalright\right
\renewcommand{\left}{\mathopen{}\mathclose\bgroup\originalleft}
\renewcommand{\right}{\aftergroup\egroup\originalright}
\usepackage{verbatim}

\makeatletter
\renewcommand*{\UrlTildeSpecial}{%
  \do\~{%
    \mbox{%
      \fontfamily{ptm}\selectfont
      \textasciitilde
    }%
  }%
}%
\let\Url@force@Tilde\UrlTildeSpecial
\makeatother

\let\OLDthebibliography\thebibliography
\renewcommand\thebibliography[1]{
  \OLDthebibliography{#1}
  \setlength{\parskip}{0pt}
  \setlength{\itemsep}{3pt plus 0.3ex}
}

\makeatother

\allowdisplaybreaks
\numberwithin{equation}{section}

\begin{document}
\title{\texorpdfstring{\vspace{-0.8cm}}{}Acyclic subgraphs of tournaments with high chromatic number}
\author{Jacob Fox \thanks{Department of Mathematics, Stanford University, Stanford, CA 94305.
Email: \href{jacobfox@stanford.edu}{\nolinkurl{jacobfox@stanford.edu}}.
Research supported by a Packard Fellowship and by NSF Award DMS-1855635.} \and Matthew Kwan \thanks{Department of Mathematics, Stanford University, Stanford, CA 94305.
Email: \href{mattkwan@stanford.edu}{\nolinkurl{mattkwan@stanford.edu}}.
Research supported in part by SNSF project 178493 and NSF Award DMS-1953990.}\and Benny Sudakov \thanks{Department of Mathematics, ETH, 8092 Z\"urich, Switzerland. Email:
\href{mailto:benny.sudakov@gmail.com} {\nolinkurl{benny.sudakov@gmail.com}}.
Research supported in part by SNSF grant 200021-196965.}}

\maketitle
\global\long\def\RR{\mathbb{R}}%
\global\long\def\FF{\mathbb{F}}%
\global\long\def\QQ{\mathbb{Q}}%
\global\long\def\E{\mathbb{E}}%
\global\long\def\Var{\operatorname{Var}}%
\global\long\def\CC{\mathbb{C}}%
\global\long\def\NN{\mathbb{N}}%
\global\long\def\ZZ{\mathbb{Z}}%
\global\long\def\TT{\mathrm{TT}}%
\global\long\def\GG{\mathbb{G}}%
\global\long\def\tallphantom{\vphantom{\sum}}%
\global\long\def\tallerphantom{\vphantom{\int}}%
\global\long\def\supp{\operatorname{supp}}%
\global\long\def\one{\mathbbm{1}}%
\global\long\def\d{\operatorname{d}}%
\global\long\def\Unif{\operatorname{Unif}}%
\global\long\def\Po{\operatorname{Po}}%
\global\long\def\Bin{\operatorname{Bin}}%
\global\long\def\Ber{\operatorname{Ber}}%
\global\long\def\Geom{\operatorname{Geom}}%
\global\long\def\Rad{\operatorname{Rad}}%
\global\long\def\floor#1{\left\lfloor #1\right\rfloor }%
\global\long\def\ceil#1{\left\lceil #1\right\rceil }%
\global\long\def\falling#1#2{\left(#1\right)_{#2}}%
\global\long\def\cond{\,\middle|\,}%
\global\long\def\su{\subseteq}%
\global\long\def\row{\operatorname{row}}%
\global\long\def\col{\operatorname{col}}%
\global\long\def\spn{\operatorname{span}}%
\global\long\def\eps{\varepsilon}%
\global\long\def\S{\mathcal{S}}%

\begin{abstract}
We prove that every $n$-vertex tournament
$G$ has an acyclic subgraph with chromatic number at least $n^{5/9-o\left(1\right)}$,
while there exists an $n$-vertex tournament $G$ whose every acyclic
subgraph has chromatic number at most $n^{3/4+o\left(1\right)}$.
This establishes in a strong form a conjecture of Nassar and Yuster and
improves on another result of theirs. Our proof combines probabilistic and spectral 
techniques together with some additional ideas. In particular,
we prove a lemma showing that every tournament with many transitive subtournaments has a large subtournament that is
almost transitive. This may be of independent interest.

\vspace{0.2cm}

\noindent\textbf{MSC Codes:} 05C15, 05C20.
\end{abstract}

\section{Introduction}

An \emph{orientation} of a graph $G$ is an assignment of a direction
to each edge. There is a long history of surprising connections between
colourings and orientations of graphs (see \cite{Hav13} for a survey).
Perhaps the most famous of these is the Gallai--Hasse--Roy--Vitaver
theorem~\cite{Gal68,Has64,Roy67,Vit62}, which states that the chromatic
number of any graph $G$ can be equivalently defined to be the minimum,
over all orientations of $G$, of the length of the longest directed
path in that orientation, plus one. This implies that for any oriented
graph $G$ with chromatic number $n$, there is a directed path with
$n$ vertices.

It is very natural to ask whether there are oriented graphs 
$H$ other than directed paths that must
necessarily appear in any oriented graph $G$ with sufficiently large
chromatic number. Since Erd\H os famously proved that there are graphs
with arbitrarily large girth and chromatic number, we can only
hope to prove results of this type when $H$ is an oriented forest.
As a far-reaching extension of the Gallai--Hasse--Roy--Vitaver
theorem, Burr~\cite{Bur80} conjectured in 1980 that any $(2k-2)$-chromatic oriented graph
contains a copy of every oriented tree on $k$ vertices.
The first general result in this direction is also due to Burr~\cite{Bur80}, who
proved that any oriented graph $G$ with chromatic number $n=\left(k-1\right)^{2}$
contains a copy of every oriented $k$-vertex tree.

Burr's conjecture has remained widely open in the 40 years since it
was proposed, even for relatively simple trees such as paths with arbitrary orientation.
A special case that has attracted a lot of attention is the case where
$G$ is a \emph{tournament}: an orientation of the complete $n$-vertex
graph. This special case of Burr's conjecture is known as \emph{Sumner's
conjecture}, and following a sequence of partial results it was was
resolved for large $n$ in a tour de force by K\"uhn, Mycroft and
Osthus~\cite{KMO11}. We remark that both Burr's and Sumner's conjectures, if true, are best-possible: there is a tournament on $(2k-3)$ vertices which does not contain a certain oriented tree on $k$ vertices.

Recently, Addario-Berry, Havet, Sales, Reed and Thomass\'e~\cite{AHSRT13}
managed to improve Burr's original bound, showing that every oriented
graph with chromatic number $n=k^{2}/2-k/2+1$ contains every oriented
tree on $k$ vertices. In the same paper, they also made the interesting
observation that if a $k$-chromatic oriented graph has no directed cycle,
then it contains every oriented tree of order $k$, and suggested
that one approach to an improved bound towards Burr's conjecture could
be to prove that every $k$-chromatic oriented digraph has an acyclic
subgraph with large chromatic number.
\begin{problem}
\label{prob:AHSRT}What is the minimum integer $f\left(k\right)$
such that every $f\left(k\right)$-chromatic oriented graph has an
acyclic $k$-chromatic subgraph?
\end{problem}

It is easy to see that the edges of any oriented graph can be partitioned
into two acyclic subgraphs (fix an ordering of the vertex set, and
consider the graph $G_{1}$ of ``forwards'' edges and the graph
$G_{2}$ of ``backwards'' edges). Using the well-known inequality
$\chi\left(G_{1}\cup G_{2}\right)\le\chi\left(G_{1}\right)\chi\left(G_{2}\right)$,
it follows that $f\left(k\right)\le k^{2}$. By taking a bit more
care with this argument, the aforementioned authors were able to prove
that $f\left(k\right)\le k^{2}-2k+2$, but it seems that new ideas
will be required to substantially improve this bound.

In the spirit of Sumner's conjecture, it is natural to consider the
restriction of \cref{prob:AHSRT} to the setting of tournaments. This was done
by Nassar and Yuster~\cite{NY19}, who asked the following.
\begin{problem}
What is the minimum integer $g\left(k\right)$ such that every $g\left(k\right)$-vertex
tournament has an acyclic $k$-chromatic subgraph?
\end{problem}

Nassar and Yuster were able to prove that $k^{8/7}/4\le g\left(k\right)\le k^{2}-\left(2-1/\sqrt{2}\right)k+2$,
and further conjectured that $g\left(k\right)=o\left(k^{2}\right)$.
Our first result proves their conjecture in a strong form.
\begin{thm}
\label{thm:main}Every $n$-vertex tournament $G$ has an acyclic
subgraph with chromatic number at least $n^{5/9-o\left(1\right)}$.
That is, $g\left(k\right)\le k^{9/5+o\left(1\right)}$.
\end{thm}

The rough idea for the proof of \cref{thm:main} is to consider a \emph{random}
acyclic subgraph of $G$, and observe that this is likely to have
high chromatic number unless $G$ contains many transitive subtournaments.
We then prove an approximate structural result (\cref{lem:structure-many-Tk})
showing that the existence of many transitive subtournaments
implies that $G$ has a large almost-transitive subtournament, from
which it is possible to deduce in a different way that $G$ has an
acyclic subgraph with high chromatic number. The details of the proof
are presented in \cref{sec:main}.

In addition, we are also able to improve Nassar and Yuster's lower
bound on $g\left(k\right)$.
\begin{thm}
\label{thm:construction}There is an $n$-vertex tournament $G$ such
that every acyclic subgraph has chromatic number at most $O\left(n^{3/4}\log n\right)$.
That is, $g\left(k\right)=\Omega\left(\left(k/\log k\right)^{4/3}\right)$.
\end{thm}

Since $f\left(k\right)\ge g\left(k\right)$, this theorem gives a
corresponding lower bound on $f\left(k\right)$, so an immediate consequence
is that one cannot hope to prove bounds stronger than $k^{4/3-o\left(1\right)}$
for Burr's conjecture via \cref{prob:AHSRT}. We prove \cref{thm:construction}
in \cref{sec:construction} using a construction due to R\"odl
and Winkler~\cite{RW89}, in which a random orientation is defined
in terms of a projective plane. This involves some nontrivial analysis
of increasing subsequences in random permutations (see \cref{subsec:ulam}).

For the sake of clarity of presentation, we omit floor and
ceiling signs where they are not crucial. All logarithms are base $e$ unless otherwise stated. 

\section{\label{sec:main}Finding an acyclic subgraph with high chromatic
number}

In this section we prove \cref{thm:main}. Note that the maximal acyclic
subgraphs of an oriented graph $G$ are all obtained by taking some
ordering $\pi$ on the vertices and considering the graph $G_{\pi}$
obtained by including the edges that are oriented ``forwards'' according
to $\pi$. Given an $n$-vertex tournament $G$, we need to show that
there is an ordering $\pi$ for which $G_{\pi}$ has high chromatic
number. First, this easily follows from consideration of a random
ordering if $G$ has few small transitive subtournaments. Let $\TT_{k}$
be the $k$-vertex transitive tournament.
\begin{lem}
\label{lem:few-Tk-alpha}If an $n$-vertex tournament $G$ has fewer
than $k!$ copies of $\TT_{k}$ then there is an ordering $\pi$ with
$\alpha\left(G_{\pi}\right)\le k$, so $\chi\left(G_{\pi}\right)\ge n/k$.
\end{lem}

\begin{proof}[Proof of \cref{lem:few-Tk-alpha}]
Consider uniformly random $\pi$. For a particular set $S$ of $k$
vertices, the probability that $S$ is an independent set in $G_{\pi}$
is either $1/k!$ if $G\left[S\right]$ is transitive, or zero if
it is not. So, by the union bound we have $\alpha\left(G_{\pi}\right)<k$
with positive probability.
\end{proof}
Given \cref{lem:few-Tk-alpha}, it suffices to consider the case where
$G$ has many copies of $\TT_{k}$. The most important ingredient is
the following lemma, showing that in this case $G$ has a large subtournament
that is almost transitive in a certain sense. Say that an $n$-vertex
tournament is \emph{$q$-almost-transitive} if there is an ordering
such that for every vertex $v$, at most $q$ edges incident to $v$
are oriented ``backwards'' according to $\pi$ (so a 0-almost transitive
tournament is transitive). 
\begin{lem}
\label{lem:structure-many-Tk}Let $k=n^{\Omega\left(1\right)}$. If
an $n$-vertex tournament $G$ has at least $k!$ copies of $\TT_{k}$
then it has a set of $n'\ge k^{2-o\left(1\right)}$ vertices which
induces an $\left(n'/k^{1-o\left(1\right)}\right)$-almost-transitive
subtournament $G'$.
\end{lem}

The other ingredient we need is the following lemma, showing that
almost-transitive tournaments have acyclic subgraphs with high chromatic
number.
\begin{lem}
\label{lem:almost-transitive-alpha}Suppose $n^{1/3}\le q\le n^{1-\Omega\left(1\right)}$.
An $n$-vertex $q$-almost-transitive tournament $G$ has an ordering
$\pi$ with $\alpha\left(G_{\pi}\right)\le n^{1/4+o\left(1\right)}q^{1/4}$,
so $\chi\left(G_{\pi}\right)\ge n^{3/4-o\left(1\right)}q^{-1/4}$.
\end{lem}

Before proving these lemmas we give the short deduction of \cref{thm:main}.
\begin{proof}[Proof of \cref{thm:main} given \cref{lem:few-Tk-alpha,lem:structure-many-Tk,lem:almost-transitive-alpha}]
Let $k=n^{4/9}$. If there are fewer than $k!$ copies of $\TT_{k}$
then \cref{lem:few-Tk-alpha} shows that there is some $\pi$ such
that $\chi\left(G_{\pi}\right)\ge n/k=n^{5/9}$ and we are done. Otherwise,
by \cref{lem:structure-many-Tk}, $G$ has an $\left(n'/k^{1-o\left(1\right)}\right)$-almost-transitive
subtournament $G'$ on $n'\ge k^{2-o\left(1\right)}$ vertices, which
by \cref{lem:almost-transitive-alpha} has an acyclic subgraph $G_{\rho}'$
such that 
\[
\chi\left(G_{\rho}'\right)\ge\left(k^{2-o\left(1\right)}\right)^{3/4-o\left(1\right)}\left(k^{1+o\left(1\right)}\right)^{-1/4}=n^{5/9-o\left(1\right)},
\]
 as desired.
\end{proof}

\subsection{Iterative alignment and refinement}

In this subsection we prove \cref{lem:structure-many-Tk}. Most
of the subsection will be spent proving the following lemma, which
should be viewed as a more technical variant of \cref{lem:structure-many-Tk}.
\begin{lem}
\label{lem:structure-many-Tk-aux} Let $n$ be sufficiently
large and let $G$ be an $n$-vertex tournament which has at least $M$
copies of $\TT_{k}$. If $q=2n\left(\log n\right)^{2}/k$, then there is $k-3k/\log n \leq k' \leq k$ such that 
$G$ has an induced $q$-almost-transitive subtournament $G'$ on at least $e^{-5}kM^{1/k}$ vertices containing
at least $e^{-3k}M$ copies of $\TT_{k'}$.
\end{lem}

The key aspect in which \cref{lem:structure-many-Tk-aux} is weaker
than \cref{lem:structure-many-Tk} is that the almost-transitivity
parameter $q=2n\left(\log n\right)^{2}/k$ depends on the number of
vertices $n$ of $G$, whereas \cref{lem:structure-many-Tk} demands
an almost-transitivity parameter with a similar dependence on the
number of vertices $n'$ of $G'$. In order to overcome this issue
we simply iterate \cref{lem:structure-many-Tk-aux}, as follows.
\begin{proof}[Proof of \cref{lem:structure-many-Tk}, given \cref{lem:structure-many-Tk-aux}]
Let $t=\sqrt{\log n}$. We obtain a sequence of subtournaments $G=G_{0}\supseteq G_{1}\supseteq\dots\supseteq G_{t}$
by iteratively applying \cref{lem:structure-many-Tk-aux} (where in
each step, we use the guarantees from the previous step on the number
of vertices and number of small transitive subtournaments). For each
$0\le i\le t$ let $n_{i}$ be the number of vertices in $G_{i}$. We let $k_i$ be the value for which we obtain a bound on the number of copies of $\TT_{k_i}$ in $G_i$. Thus, $k_{0}=k$ and for $1\le i\le t$ we have $k_{i} \geq k_{i-1}\left(1-3/\log n_{i-1}\right)$. Let $q_{i}=2n_{i-1}\left(\log n_{i-1}\right)^{2}/k_{i-1}$. Then,
each $G_{i}$ is $q_{i}$-almost-transitive, has at least $M_{i}:=e^{-3\left(k_{0}+\dots+k_{i-1}\right)}k!\ge e^{-3ik}k!$
copies of $\TT_{k_{i}}$, and for $i>0$, has $n_{i}\ge e^{-5}k_{i-1}\left(M_{i-1}\right)^{1/k_{i-1}}$
vertices.

Combining our lower bounds on $n_i$ and $M_{i-1}$, and using Stirling's approximation, we obtain $n_i\ge \Omega(e^{-3i} k_{i-1} k)$. Recalling that $k_{i} \geq k_{i-1}\left(1-3/\log n_{i-1}\right)$ and that $k=n^{\Omega(1)}$, it is straightforward to inductively prove that $n_i\ge k^{2-o(1)}$ and $k_i \geq ke^{-O(i/\log n)} =(1-o(1))k$ for every $i\le t =\sqrt{\log n}$. It follows that each $q_{i}=n_{i-1}k^{o\left(1\right)-1}$.

There is some $i$ such that $\log n_i - \log n_{i-1}=\log\left(n_{i}/n_{i-1}\right)\le\left(1/t\right)\log n\le\sqrt{\log n}$,
meaning that $n_{i}/n_{i-1}=k^{o\left(1\right)}$, so $q_i=n_i/k^{1-o(1)}$. Then $G_{i}$ satisfies
the required properties.
\end{proof}
To prove \cref{lem:structure-many-Tk} it therefore suffices to prove
\cref{lem:structure-many-Tk-aux}. We need some very simple auxiliary
lemmas.
\begin{lem}
\label{lem:min-degree-subgraph}Let $H$ be a hypergraph with $m$ nonempty edges and $n$ vertices.
Then it has an induced subhypergraph with minimum degree at least $m/n$.
\end{lem}

\begin{proof}
Iteratively delete a vertex of degree less than $m/n$ together with the edges touching it, until no such vertex remains. 
In total we deleted fewer than $(m/n) \cdot n=m$ edges, so we are left with a nonempty induced
subhypergraph that satisfies the assertion of the lemma.
\end{proof}
\begin{lem}
\label{lem:tournament-large-degree}Every $n$-vertex tournament $G$
has a vertex whose indegree and outdegree are both at least $(n-2)/4$.
\end{lem}

\begin{proof}
Let $A$ (respectively $B$) be the set of all vertices whose indegree
(respectively, outdegree) is at least $(n-2)/4$. For every vertex,
the sum of its indegree and outdegree is exactly $n-1 \geq 2\left(n-2\right)/4$,
so every vertex is in $A$ or $B$. Without loss of generality suppose
$\left|A\right|\ge\left|B\right|$. Then observe that the induced
subtournament $G\left[A\right]$ has average outdegree $\left(\left|A\right|-1\right)/2\ge\left(n/2-1\right)/2= (n-2)/4$,
so it has a vertex whose outdegree is at least this average. This
vertex has the desired property.
\end{proof}
Now we prove \cref{lem:structure-many-Tk-aux}.
\begin{proof}[Proof of \cref{lem:structure-many-Tk-aux}]
We iteratively build sequences of vertex sets 
\[
\emptyset=W_{0}\subset W_{1}\subset\dots\subset W_{t}\subset V(G),\quad V\left(G\right)=V_{0}'\supseteq V_0\supset V_{1}'\supseteq V_1\supset\dots V_t'\supseteq V_{t}
\]
for some $t< 3k/\log n$, in such a way that each $\left|W_{i}\right|=i$ 
and each $V_{i}'\cap W_{i},V_{i}\cap W_{i}=\emptyset$. The idea is that the $W_i$ will be sets of vertices each inducing a copy of $\TT_i$, and $V_{i}',V_i$
will be sets of vertices which can be used to extend $W_{i}$ into
a copy of $\TT_{k}$ in many different ways. We will add vertices one-by-one
to form the $W_{i}$, in such a way that $G\left[V_{i}'\right]$, $G\left[V_{i}\right]$
gradually get ``closer to being transitive'', and then we will take $G'=G[V_t]$. In order to explain how to iteratively build our desired sets, we need to make some definitions.
\begin{itemize}
\item Let $H_{i}'$ be the $\left(k-i\right)$-uniform hypergraph with vertex set $V_{i}'$ whose hyperedges are those $\left(k-i\right)$-vertex subsets $S$ for which $G\left[S\cup W_{i}\right]$ forms a copy of $\TT_{k}$.
\item Let $N_{i}'$ be the number of edges in $H_{i}'$. Equivalently, $N_{i}'$
is the number of copies of $\TT_{k}$ in $G\left[V_{i}' \cup W_i\right]$ which
include all vertices of $W_{i}$. In particular, $N_{0}'\ge M$.
\item Define $V_i$ from $V_i'$ by letting $H_{i}=H_{i}'\left[V_{i}\right]$ be a nonempty induced subgraph
of $H_{i}'$ with minimum degree at least $N_{i}'/\left|V_{i}'\right|$,
which exists by \cref{lem:min-degree-subgraph}. Let $N_{i}\ge N_{i}'/\left|V_{i}'\right|\ge N_{i}'/n$
be the number of edges in $H_{i}$.
\item For each $i$ and each $v\in V_{i}$, note that $W_{i}\cup\left\{ v\right\}$
induces a transitive subtournament on $i+1$ vertices (because $v$
is contained in at least one edge of $H_{i}$). In this transitive
tournament, there are $i+1$ possibilities for the position of $v$
relative to the vertices in $W_{i}$. Partition $V_{i}$ into subsets
$V_{i,0}\cup\dots\cup V_{i,i}$ according to these $i+1$ possibilities:
let $V_{i,j}$ be the set of vertices $v\in V_{i}$ which have indegree
exactly $j$ in $G\left[W_{i}\cup\left\{ v\right\} \right]$.
\item For vertices $x\in V_{i,j}$ and $y\in V_{i,q}$ with $j<q$ we
say that $x$ and $y$ are \emph{incompatible} if the edge between
$x$ and $y$ is oriented towards $x$ (so then there can be no transitive
tournament containing $W_{i}\cup\left\{ x,y\right\} $).
\end{itemize}
Now, we will choose $w\in V_{i}$ (which we will then add to $W_{i}$
to form $W_{i+1}$) according to one of the following two procedures.
\begin{itemize}
\item \textbf{Refinement.} Consider a largest $V_{i,j}$. By \cref{lem:tournament-large-degree},
$G\left[V_{i,j}\right]$ has a vertex with indegree and outdegree
at least $\left(\left|V_{i,j}\right|-2\right)/4 \geq \left|V_{i,j}\right|/4-1$, which we take as $w$.
\item \textbf{Alignment.} Let $\varepsilon=\left(\log n\right)^{2}/k$.
If there is a vertex which is incompatible with $\varepsilon\left|V_{i}\right|$
other vertices of $V_{i}$, then take such a vertex as $w$.
\end{itemize}
Note that alignment may not always be possible, but refinement is
always possible. After choosing $w\in V_{i}'$ according to one of
the above two procedures, we then set $W_{i+1}=W_{i}\cup\left\{ w\right\} $
and let $V_{i+1}'$ be obtained from $V_{i}$ by deleting $w$ and
all vertices incompatible with it.

The purpose of the alignment steps is to ensure that most edges between the $V_{i,j}$ are oriented in a single direction, and the purpose of the refinement steps is to ensure that the $V_{i,j}$ are not too large (so there are not very many edges inside the $V_{i,j}$, whose orientations we do not control). We make some simple observations about both procedures.
\begin{claim*}
No matter how we decide whether to do alignments or refinements in
our $t<3k/\log n$ steps, the following always hold.
\begin{enumerate}
\item [(a)]If step $i$ is an alignment step, then $\left|V_{i+1}'\right|\le\left(1-\varepsilon\right)\left|V_{i}'\right|$.
\item [(b)]$N_{t}\ge e^{-3k}M$.
\item [(c)]$\left|V_{t}\right|\ge e^{-5}kM^{1/k}$.
\end{enumerate}
\end{claim*}
\begin{proof}
First, (a) follows from the definition of an alignment step. Next,
observe that each $H_{i+1}'$ is a \emph{link hypergraph} of $H_{i}$,
obtained by taking all edges incident to a vertex $w\in H_{i}$ and
deleting $w$ from each of them. Since every vertex in $H_{i}$ has
degree at least $N_{i}'/n$, it follows
that $N_{i+1}'\ge N_{i}'/n$. Iterating this $t<3k/\log n$ times, then
using the fact that $N_{t}\ge N_{t}'/n$, yields (b). Finally, since
$H_{t}$ is a $\left(k-t\right)$-uniform hypergraph, note that
\[
N_{t}\le\binom{\left|V_{t}\right|}{k-t}\le\left(\frac{e\left|V_{t}\right|}{k-t}\right)^{k-t},
\]
so
\[
\left|V_{t}\right|\ge\left(\frac{k-t}{e}\right)\left(N_{t}\right)^{1/\left(k-t\right)}.
\]
Combining this with (b) yields (c), for large $n$.
\end{proof}
We will end up taking $G\left[V_{t}\right]$ as our subtournament
$G'$, so properties (b) and (c) in the above claim ensure that it
has enough vertices and enough copies of $\TT_{k'}$, for some $k'\ge k-3k/\log n$. We just need to
decide how to choose whether to do alignments and refinements, in
such a way that $G\left[V_{t}\right]$ ends up being $q$-almost-transitive.
These choices are very simple: we first do $s:=k/\log n$ refinement
steps, then do alignment steps as long as is possible. By (a), each
alignment step decreases $V_{i}'$ by a factor of $\left(1-\varepsilon\right)$,
so it is not possible to perform more than $\log_{1/(1-\varepsilon)}n<\left(2/\varepsilon\right)\log n=2k/\log n$
alignment steps (using the inequality $-\log\left(1-\varepsilon\right)>\varepsilon/2$
for small $\varepsilon>0$). Thus, the total number of steps will
be $t<3k/\log n$, as we have been assuming.
\begin{claim*}
For $i\ge s=k/\log n$, we have
\[
\max_{j}\,\left|V_{i,j}\right|\le\frac{4n}{s+1}.
\]
\end{claim*}
\begin{proof}
For $i\le s$, the way the parts $V_{i,j}$ evolve is as follows.
At each step, we ``split'' a largest part $V_{i,j}$ into two sets,
each of size at most $3\left|V_{i,j}\right|/4$ (removing a single
element $w$), and then we delete some further elements in each part.
For steps $i>s$, the parts $V_{i,j}$ only get smaller, so it suffices
to consider $\max_{j}\,\left|V_{s,j}\right|$.

Consider the following alternative procedure to iteratively build
a partition of $\left\{ 1,\dots,n\right\} $. Start with the trivial
partition $\mathcal{P}_{0}$ into one part, and for each step $i$
consider a largest part $A\in\mathcal{P}_{i}$, split it into two
parts of sizes $\ceil{\left|A\right|/4}$ and $\floor{3\left|A\right|/4}$,
and call the resulting partition $\mathcal{P}_{i+1}$. This procedure
``dominates'' the procedure used to build the $V_{i,j}$, in the
sense that for each $i\le s$ we have 
\[
\max_{j}\, \left|V_{i,j}\right|\le\max_{A\in\mathcal{P}_{i}}\,\left|A\right|.
\]
But note that in this second procedure we always maintain the property
\[
\max_{A\in\mathcal{P}_{i}}\, \left|A\right|\le4\min_{A\in\mathcal{P}_{i}}\, \left|A\right|,
\]
so $\max_{A\in\mathcal{P}_{s}}\,\left|A\right|$ is at most four times the average size of the $A\in\mathcal{P}_{s}$, which
is $n/\left(s+1\right)$.
\end{proof}
Now, after all our alignment steps have finished, in $V_t$ there is no
vertex which is incompatible with more than $\varepsilon\left|V_{t}\right|$
other vertices of $V_{t}$. Consider an arbitrary ordering of
$V_{t}$ such that for $j<q$, all the elements of $V_{t,j}$ appear
before all the elements of $V_{t,q}$. In $G\left[V_{t}\right]$,
for any vertex $v$ there are a total of at most $\varepsilon\left|V_{t}\right|+4n/\left(s+1\right)\le2n\left(\log n\right)^{2}/k=q$
edges adjacent to $v$ which are oriented ``against'' our chosen
ordering (for large $n$). So, $G\left[V_{t}\right]$ is $q$-almost-transitive,
as desired.
\end{proof}

\subsection{Biased random orderings}

In this subsection we prove \cref{lem:almost-transitive-alpha}. If
$G$ is $q$-almost-transitive then we can choose $\rho$ to be the
ordering defining almost-transitivity, yielding a very dense graph
$G_{\rho}$ whose complement has maximum degree at most $q$. One
can give a lower bound on $\chi\left(G_{\rho}\right)$ simply using
the number of edges of $G_{\rho}$, but to prove \cref{lem:almost-transitive-alpha}
we need to do better than this. We will consider a random perturbation
$\pi$ of $\rho$, which has the same ``large-scale'' structure
as $\rho$, but on a ``small scale'' it resembles a uniformly random
permutation.

We first need a simple auxiliary lemma about uniform random permutations.
\begin{lem}
\label{lem:aux-longest-increasing}Let $G$ be an $n$-vertex tournament
and let $\pi$ be a random permutation of its vertex set. Then with
probability $1-n^{-\omega\left(1\right)}$ we have $\alpha\left(G_{\pi}\right)=O\left(\sqrt{n}\right)$.
\end{lem}

\begin{proof}
For any $k$, the expected number of size-$k$ independent sets in
$G_{\pi}$ is $N_{G}/k!\le\binom{n}{k}/k!$, where $N_{G}$ is the
number of copies of $\TT_{k}$ in $G$. If $k$ is a large multiple
of $\sqrt{n}$ this number is $n^{-\omega\left(1\right)}$.
\end{proof}
Now we prove \cref{lem:almost-transitive-alpha}.
\begin{proof}[Proof of \cref{lem:almost-transitive-alpha}]
Let $s=\sqrt{n/q}$. Consider the ordering $\rho$ defining almost-transitivity
(so for each vertex $v$, only $q$ edges incident to $v$ are oriented
``against'' $\rho$, to some set of vertices which we call $B\left(v\right)$).
Starting from the ordering $\rho$, we divide the vertex set into
$s$ contiguous blocks $A_{1},\dots,A_{s}$ of size $n/s$, and then
randomly permute the indices within each block, to obtain a random
ordering $\pi$ (retaining the large-scale structure of $\rho$, but
introducing some ``local'' randomness).

By \cref{lem:aux-longest-increasing} and the union bound, with positive
probability $\pi$ has the property that for each $1\le i\le s$ and
each vertex $v$, we have 
\begin{align}
\alpha\left(G_{\pi}\left[A_{i}\right]\right) & =O\left(\sqrt{n/s}\right),\label{eq:alpha-1}\\
\alpha\left(G\left[A_{i}\cap B\left(v\right)\right]\right) & \le n^{o\left(1\right)}+O\left(\sqrt{\left|A_{i}\cap B\left(v\right)\right|}\right).\label{eq:alpha-2}
\end{align}
Consider such an outcome of $\pi$.

Now, consider an independent set $S$ in $G_{\pi}$. Let $v$ be the
first vertex in $S$ (according to the ordering $\pi$), and let $A_{i}$
be the block that contains $v$. By the definition of $\pi$, we must
have $S\subseteq A_{i}\cup B\left(v\right)$. Next, \cref{eq:alpha-1}
implies that $\alpha\left(G_{\pi}\left[S\cap A_{i}\right]\right)=O\left(\sqrt{n/s}\right)$. Hence
by \cref{eq:alpha-2} and convexity of the function $f(x)=\sqrt{x}$, we have
\[
\alpha\left(G_{\pi}\left[S\backslash A_{i}\right]\right)\le\sum_{j}\left(n^{o\left(1\right)}+O\left(\sqrt{\left|A_{j}\cap B\left(v\right)\right|}\right)\right)\le sn^{o\left(1\right)}+O\left(\sqrt{qs}\right),
\]
where the sum is over all $j\ne i$ with $A_{i}\cap B\left(v\right)\ne\emptyset$.
It follows that $\alpha\left(G_{\pi}\left[S\right]\right)\le\alpha\left(G_{\pi}\left[A_{i}\right]\right)+\alpha\left(G_{\pi}\left[S\backslash A_{i}\right]\right)\le O\left(\sqrt{n/s}+\sqrt{qs}\right)+sn^{o\left(1\right)}\le n^{1/4+o\left(1\right)}q^{1/4}$,
recalling that $q\ge n^{1/3}$ (therefore $s\le n^{1/4}q^{1/4}$).
\end{proof}

\section{\label{sec:construction}A tournament with easily-colourable acyclic
subgraphs}

In this section we prove \cref{thm:construction}. We will study the
following construction, which was previously introduced by R\"odl
and Winkler~\cite{RW89} (see also \cite{APS01}).
\begin{defn}
\label{def:projective-plane}Let $G$ be a tournament obtained as
follows. First, consider a projective plane $P$ of order $t-1$,
with $t^{2}-t+1$ lines and $t^{2}-t+1$ points, $t$ points on every
line and $t$ lines going through every point. Then, $G$ will be
a tournament on $\left(t^{2}-t+1\right)k$ vertices, where for each
point of the projective plane there are $k$ associated vertices of
$G$ (so, our vertices are divided into $t^{2}-t+1$ ``buckets'').
For the edges inside the buckets, choose their orientation arbitrarily.
For the edges between the buckets, we do the following. For each line
of $P$, consider the $kt$ corresponding vertices of $G$. Choose
a random ordering of these vertices, and orient all remaining edges
between these vertices according to the ordering. Since every pair
of points in $P$ are contained in exactly one line, we can do this
independently for each line.
\end{defn}

Now, the following lemma will imply \cref{thm:construction} almost
immediately. Recall from the previous section that for an oriented
graph $G$, we define $G_{\pi}$ to be the subgraph obtained by including
the edges that are oriented ``forwards'' according to $\pi$.
\begin{lem}
\label{lem:independent-set}There is an absolute constant $C$ such
that the following holds. Let $G$ be as in \cref{def:projective-plane},
with $k=t^{2}$ and $t$ sufficiently large. With positive probability,
$G$ has the following property. For any ordering $\pi$ of the vertices
of $G$ and any subset $X$ of at least $100t^{3}\log t$ vertices
of $G$, we have $\alpha\left(G_{\pi}\left[X\right]\right)\ge\left|X\right|/\left(16t^{3}\right)$.
\end{lem}

Before turning to the proof of \cref{lem:independent-set} we give
the short deduction of \cref{thm:construction}.
\begin{proof}[Proof of \cref{thm:construction}]
For any $n$, by Bertrand's postulate we can find a prime $t-1=\Theta\left(n^{1/4}\right)$,
with $\left(t^{2}-t+1\right)t^{2}\ge n$. So, there is a projective
plane of order $t-1$ and we can consider $G$ as in \cref{def:projective-plane},
with $k=t^{2}$. Consider a particular outcome of $G$ satisfying
the condition in \cref{lem:independent-set}. Then, consider any acyclic
subgraph $G'$ of $G$, which is a subgraph of some $G_{\pi}$ (simply
take $\pi$ to be an ordering corresponding to a transitive closure
of $G'$).

Now, we construct a proper colouring of $G_{\pi}$ as follows. First
greedily take maximum independent sets as colour classes, until there
are fewer than $100t^{3}\log t$ vertices remaining. Then, give each
remaining vertex its own unique colour. Let $\ell$ be the number
of colours coming from the greedy phase of this procedure, so $k(t^2-t+1)\left(1-1/\left(16t^{3}\right)\right)^{\ell-1}\ge100t^{3}\log t$
by the property in \cref{lem:independent-set}, and therefore $\ell=O\left(t^{3}\log t\right)$.
We deduce that $\chi\left(G_{\pi}\right)=\ell+100t^{3}\log t=O\left(t^{3}\log t\right)$.

Finally, we can simply delete a few vertices from $G$ to obtain a
tournament on exactly $n$ vertices with the desired properties, using the fact that the chromatic number of a subgraph is at most 
the chromatic number of the graph. 
\end{proof}
The first ingredient in the proof of \cref{lem:independent-set} is
the fact that, in $G$, all large sets of vertices are well-distributed
between lines of the underlying projective plane.
\begin{lem}
\label{cor:few-low-degree}Let $G$ be as in \cref{def:projective-plane},
and let $X$ be a set of at least $9kt$ vertices of $G$. Then at
most $t^{2}/2$ lines of the projective plane underlying $G$ contain
fewer than $\left|X\right|/\left(2t\right)$ vertices of $X$.
\end{lem}

We also need the following lemma giving strong bounds on the probability
that a random permutation fails to contain a long increasing subsequence,
where each index in the subsequence comes from a different ``bucket''.
\begin{lem}
\label{lem:monotone-subsequence-lemma}Divide the interval $\left\{ 1,\dots,m\right\} $
into ``buckets'' of size at most $k$, where $m\le k^{2}$. Consider
a random permutation $\sigma$ of $\left\{ 1,\dots,m\right\} $. Let
$\ell=m/\left(8k\right)$. With probability at least $1-e^{-m/24}$
there is a sequence $1\le i_{1}<\dots<i_{\ell}\le m$ with each $i_{j}$
coming from a different bucket, such that $\sigma\left(i_{i}\right)<\dots<\sigma\left(i_{\ell}\right)$.
\end{lem}

We prove the above two lemmas at the end of the section, but first
we deduce \cref{lem:independent-set}.
\begin{proof}[Proof of \cref{lem:independent-set}]
Consider any ordering $\pi$ and vertex set $X$, with $\left|X\right|\ge100t^{3}\log t$
as in the lemma statement. By \cref{cor:few-low-degree} there are
at least $t^2-t+1 -t^2/2= t^{2}/2-t+1$ lines which contain at least $\left|X\right|/\left(2t\right)$
vertices of $X$. For any such line $\Lambda\subseteq V\left(G\right)$,
let $\mathcal{E}_{\Lambda}$ be the event that there is an independent
set of size $\left|X\right|/\left(16t^{3}\right)$ in $G_{\pi}\left[\Lambda\cap X\right]$
with each vertex coming from a different bucket. Note that each of
the $\mathcal{E}_{\Lambda}$ are mutually independent, and by \cref{lem:monotone-subsequence-lemma}
each $\Pr\left(\mathcal{E}_{\Lambda}\right)\ge1-e^{-\left|X\right|/\left(48t\right)}$.
So, the probability no $\mathcal{E}_{\Lambda}$ holds is at most $e^{-(t^{2}/2-t+1)\left|X\right|/\left(48t\right)} \leq e^{-t\left|X\right|/97}$.
This probability is small enough to take the union bound over all
$\left|V\left(G\right)\right|!\le e^{t^{4}\log t}$ orderings $\pi$
and at most $2^{\left|V\left(G\right)\right|}=e^{O\left(t^{4}\right)}$ subsets
$X$.
\end{proof}

\subsection{Quasirandomness}

In this section we prove \cref{cor:few-low-degree}. We will want a
bipartite version of the expander mixing lemma, due to Haemers~\cite[Theorem~5.1]{Hae95}
(see also \cite[Lemma~8]{DSV12} for a version with notation closer
to what we use here). Note that, for a bipartite graph $H$, the nonzero
singular values of the bipartite adjacency matrix of $H$ are in correspondence
with the positive eigenvalues of the (non-bipartite) adjacency matrix
of $H$.
\begin{lem}
\label{lem:expander-mixing}Let $H$ be a biregular bipartite graph
with parts $A$ and $B$, where every vertex in $A$ has degree $a$
and every vertex in $B$ has degree $b$. Let $M$ be the $\left|A\right|\times\left|B\right|$
bipartite adjacency matrix of $H$, and let $\sigma_{1}\ge\sigma_{2}\ge\dots$
be the singular values of $M$. Then for any sets $X\subseteq A$
and $Y\subseteq B$, we have
\[
\left|e\left(X,Y\right)-\frac{\sqrt{ab}}{\sqrt{\left|A\right|\left|B\right|}}\left|X\right|\left|Y\right|\right|\le\sigma_{2}\sqrt{\left|X\right|\left|Y\right|}.
\]
\end{lem}

To apply \cref{lem:expander-mixing} we will want to study the singular
values of a bipartite graph related to \cref{def:projective-plane}.
\begin{lem}
\label{lem:pp-singular-values}Let $G$ be as in \cref{def:projective-plane},
and let $H$ be the bipartite graph with parts $A$ and $B$ defined
as follows. Let $A$ be the set of vertices of $G$, let $B$ be the
set of lines in the projective plane underlying $G$, and put an edge
between a vertex and a line if the vertex is in a bucket corresponding
to a point on that line. Let $M$ be the bipartite adjacency matrix
of $H$. Then the nonzero singular values of $M$ are $\sqrt{k}t$
with multiplicity 1, and $\sqrt{k\left(t-1\right)}$ with multiplicity
$t^{2}-t$.
\end{lem}

\begin{proof}
Let $L$ be the $\left(t^{2}-t+1\right)\times\left(t^{2}-t+1\right)$
point-line incidence matrix associated with the projective plane underlying
$G$. Then $M$ can be represented as the Kronecker (tensor) product
$K\otimes L$, where $K$ is the $k\times1$ all-ones matrix. Note
that $L^{T}L$ has $t$ for all diagonal entries, and $1$ for all
off-diagonal entries, because in an order-$\left(t+1\right)$ projective
plane, every point lies in $t$ lines, and every pair of points lie
in exactly one line. $L^{T}L$ has eigenvalues $t+\left(t^{2}-t+1\right)-1=t^{2}$
(with eigenvector $\left(1,\dots,1\right)$) and $t-1$ (with eigenvectors
of the form $\left(-1,0,\dots,0,1,0\dots\right)$). Also, $K^{T}K$
is the $1\times1$ matrix whose entry is $k$. So, $M^{T}M=\left(K^{T}\otimes L^{T}\right)\left(K\otimes L\right)=\left(K^{T}K\otimes L^{T}L\right)=kL^{T}L$,
and the desired claim about singular values follows.
\end{proof}
Finally, combining the above two lemmas yields \cref{cor:few-low-degree},
as follows.
\begin{proof}[Proof of \cref{cor:few-low-degree}]
Consider the bipartite graph $H$ in \cref{lem:pp-singular-values}.
Note that this graph has parts of size $t^2-t+1$ and  $k(t^2-t+1)$. All the degrees in the first part are $kt$ and all the degrees in the second part are $t$. 
Suppose more than $t^{2}/2$ lines contain fewer than $\left|X\right|/\left(2t\right)$
vertices. Then in $H$ we have $X\subseteq A$ and $Y\subseteq B$
with $\left|X\right|\ge9kt$, $\left|Y\right|>t^{2}/2$ and $e\left(X,Y\right)<\left|Y\right|\left|X\right|/\left(2t\right)$.
Then by \cref{lem:expander-mixing} (with $\sigma_{2}=\sqrt{k\left(t-1\right)}$,
coming from \cref{lem:pp-singular-values}),
\[
(1+o(1))\frac{|X|Y|}{2t} \leq \left|e\left(X,Y\right)-\frac{1+o(1)}{t}\left|X\right|\left|Y\right|\right|\le\sqrt{k\left(t-1\right)} \sqrt{\left|X\right|\left|Y\right|},
\]  
contradiction.
\end{proof}

\subsection{\label{subsec:ulam}Long increasing subsequences}
\begin{proof}[Proof of \cref{lem:monotone-subsequence-lemma}]
First, consider independent random variables $\alpha_{1},\dots,\alpha_{m}$
each uniform in the interval $\left(0,1\right]$ (these are distinct
with probability 1), and note that we can define a uniformly random
permutation $\sigma:\left\{ 1,\dots,m\right\} \to\left\{ 1,\dots,m\right\} $
by taking $\alpha_{\sigma\left(1\right)}<\dots<\alpha_{\sigma\left(m\right)}$.
Then, divide the interval $\left(0,1\right]$ into $m/\left(2k\right)$
equal-sized sub-intervals $I_{1},\dots,I_{m/\left(2k\right)}$ (so
$I_{y}=\left(2\left(y-1\right)k/m,2yk/m\right]$ for each $y$). We
then define a random zero-one matrix $M\in\left\{ 0,1\right\} ^{\left(m/(2k)\right)\times m}$,
where $M\left(y,x\right)=1$ if $\alpha_{x}\in I_{y}$, and $M\left(y,x\right)=0$
otherwise. Note that the columns of this matrix are independent, each
uniform among all zero-one vectors with exactly one ``1''.

Next, in addition to the existing division into buckets, we also divide
$\left\{ 1,\dots,m\right\} $ into $m/k$ (discrete) intervals we
call ``slices'': for $1\le r\le m/k$, let $S_{r}=\left\{ \left(r-1\right)k+1,\dots,rk\right\} $.
For $0\le q<m/\left(2k\right)$, define the ``$q$-th diagonal strip''
$D_{q}=\left\{ \left(y,x\right):1\le y\le m/\left(2k\right),x\in S_{y+q}\right\} $
(which we interpret as a set of positions in our matrix $M$), and
note that these strips each have size $m/2$ and are disjoint for
different $q$.

Now, for each $0\le q<m/\left(2k\right)$ consider the following procedure
to find a suitable sequence $x_{1}<\dots<x_{\ell}$ (which we call
``phase $q$''). We order the positions in $D_{q}$ from left to
right, and scan through the corresponding entries of $M$ until we
see a ``1''. We take the corresponding column as $x_{1}$, then
we continue scanning, skipping all entries that share a bucket or
a slice with $x_{1}$ (without exposing their values) until we find
another ``1'', whose column we take as $x_{2}$. We continue in
this fashion, continually scanning along $D_{q}$ and skipping all
elements that share a bucket or a slice with a previously selected
element, until we have constructed a sequence $x_{1}<\dots<x_{\ell}$
of length $\ell$ or we have reached the end of the entries indexed
by $D_{q}$. In the latter case we say that the procedure fails, and
denote the corresponding event by $\mathcal{E}_{q}$. Note that during
phase $q$ we exposed only some subset of the entries indexed by $D_{q}$
(though since the entries are dependent, we indirectly revealed some
partial information about entries in other diagonal strips). Among
the entries we exposed, at most $\ell$ of these are ``1''s. See \cref{fig:phases} below.

\vspace{10pt}
\begin{figure}[h]
\newcommand\x{\cellcolor{black!10}}
\newcommand\y{\cellcolor{black!30}}
\[
\left(\begin{array}{cccccccccccccccccc}
 &  &  &  &  &  &  &  &  & 0\x & 1\x & *\x & *\y & 1\y & \phantom{*}\y &  & \cdots\\
 &  &  &  &  &  & 0\x & *\x & 0\x & 0\y & 0\y & 1\y &  &  &  & \phantom{0} & \phantom{0} & \phantom{0}\\
 &  &  & 1\x & *\x & *\x & 0\y & 0\y & *\y\\
0\x & 0\x & 0\x & 0\y & 1\y & *\y
\end{array}\right)
\]

\caption{\label{fig:phases}An example outcome of the first two phases. In this case $k=3$, $m=24$ and we are looking for an increasing sequence of length $\ell=3$ (these values of $k,m,\ell$ were chosen to yield a simple picture, but note that they do not actually satisfy the conditions in the lemma statement). First, in phase 0 we searched from left to right through the light grey cells in $D_{0}$. We first saw a ``1'' in column 4, meaning that we take $x_{1}=4$. Then we continued searching, skipping columns that share a bucket or slice with $x_{1}$ (the skipped cells are labelled with ``$*$''; here 4 and 8 share a bucket). We eventually found a second ``1'' in column 11, so we take $x_2=11$, but we were not able to find a third ``1'' so the first phase failed ($\mathcal E_0$ occurred). In phase 1 we scanned through the dark grey cells in $D_{1}$, in the same way. This phase succeeded, finding $(x_{1},x_2,x_3)=(5,12,14)$.}
\end{figure}

Observe that if any of the $\mathcal{E}_{q}$ do not occur (that is,
some phase $q$ succeeds), then the sequence $x_{1}<\dots<x_{\ell}$
produced in phase $q$ satisfies the conditions of the lemma. It therefore
suffices to show that $\Pr\left(\bigcap_{q}\mathcal{E}_{q}\right)\le e^{-m/24}$.
We will accomplish this by showing that $\Pr\left(\mathcal{E}_{q}\cond\mathcal{E}_{0}\cup\dots\cup\mathcal{E}_{q-1}\right)\le e^{-k/12}$
for each of the $m/\left(2k\right)$ choices of $q$. So, fix some
$q$, and condition on any outcome of phases $0$ to $q-1$.

Now, suppose that phase $q$ failed (meaning $\mathcal{E}_{q}$ occurred).
Since each bucket and slice has size at most $k$, during the scanning
procedure we skipped fewer than $2\ell k$ entries (out of the $m/2$
total entries indexed by $D_{q}$) for sharing a bucket or slice with
a previously chosen element. Also, since each phase exposes at at
most $\ell$ ``1''s, there are at most $\ell q$ columns in which
we saw a ``1'' in previous phases.

Now, for each $\left(y,x\right)\in D_{q}$ for which we have not already
seen a ``1'' in column $x$ in a previous phase, the probability
that $M_{y,x}=1$ is at least $2k/m$ (it could be greater than this
if we have already seen a ``0'' in column $x$ in a previous phase).
So, $\Pr\left(\mathcal{E}_{q}\cond\mathcal{E}_{0}\cup\dots\cup\mathcal{E}_{q-1}\right)$
is upper-bounded by $\Pr\left(X<\ell\right)$, where $X$ has the
binomial distribution $\Bin\left(m/2-2\ell k-\ell q,2k/m\right)$.
Recalling that $\ell=m/\left(8k\right)$, $m\le k^{2}$ and $q<m/\left(2k\right)$
(so $\E X\ge3k/8$ and $3k/8-\ell\ge k/4$), this probability is at
most 
\[
\exp\left(-\frac{\left(k/4\right)^{2}}{2\left(3k/8\right)}\right)=e^{-k/12}
\]
 by a Chernoff bound (see for example \cite[Theorem~2.1]{JLR00}).
\end{proof}

\section{Concluding remarks}

Recall that $g\left(k\right)$ is the minimum integer such that every $g\left(k\right)$-vertex
tournament has an acyclic $k$-chromatic subgraph. Here we proved
that $k^{4/3-o\left(1\right)}\le g\left(k\right)\le k^{9/5+o\left(1\right)}$.
We further suspect that the lower bound may be closer to the truth and propose the following conjecture.
\begin{conjecture}
Every $n$-vertex tournament has an acyclic subgraph with chromatic
number at least $n^{3/4-o\left(1\right)}$. That is, $g\left(k\right)\le k^{4/3+o\left(1\right)}$.
\end{conjecture}

Our proof of the lower bound $g\left(k\right)\ge k^{4/3-o\left(1\right)}$
is based on a certain randomized construction involving a projective plane,
but we wonder if the following simpler construction may also attain
the same bound. Suppose that $n=q^{2}$ is a perfect square, and consider
the tournament on the vertex set $\left\{ 1,\dots,n\right\} $ where,
for $i<j$, the edge $ij$ is oriented from $i$ to $j$ unless $i-j$
is divisible by $q$, in which case it is oriented from $j$ to $i$.
This particular tournament was used by Nassar and Yuster 
to prove their lower bound $g\left(k\right)\ge n^{8/7}/4$.
\vspace{5pt}

{\bf Acknowledgments.} The second author would like to thank Asaf Ferber for helpful discussions about multiple-exposure arguments for random permutations (which turned out to be helpful for the proof of \cref{lem:monotone-subsequence-lemma}). The third author would like to thank Noga Alon for bringing the problem of
Nassar and Yuster to his attention.


\begin{thebibliography}{10}

\bibitem{AHSRT13}
L.~Addario-Berry, F.~Havet, C.~L. Sales, B.~Reed, and S.~Thomass\'{e},
  \emph{Oriented trees in digraphs}, Discrete Math. \textbf{313} (2013),
  967--974.

\bibitem{APS01}
N.~Alon, J.~Pach, and J.~Solymosi, \emph{Ramsey-type theorems with forbidden
  subgraphs}, Combinatorica \textbf{21} (2001), 155--170, Paul Erd\H{o}s
  and his mathematics (Budapest, 1999).

\bibitem{Bur80}
S.~A. Burr, \emph{Subtrees of directed graphs and hypergraphs}, Congr. Numer.
  \textbf{28} (1980), 227--239.

\bibitem{DSV12}
S.~De~Winter, J.~Schillewaert, and J.~Verstraete, \emph{Large incidence-free
  sets in geometries}, Electron. J. Combin. \textbf{19} (2012), Paper
  24, 16.

\bibitem{Gal68}
T.~Gallai, \emph{On directed paths and circuits}, Theory of {G}raphs ({P}roc.
  {C}olloq., {T}ihany, 1966), Academic Press, New York, 1968, pp.~115--118.

\bibitem{Hae95}
W.~H. Haemers, \emph{Interlacing eigenvalues and graphs}, Linear Algebra Appl.
  \textbf{226/228} (1995), 593--616.

\bibitem{Has64}
M.~Hasse, \emph{Zur algebraischen {B}egr\"{u}ndung der {G}raphentheorie. {I}},
  Math. Nachr. \textbf{28} (1964/65), 275--290.

\bibitem{Hav13}
F.~Havet, \emph{Orientations and colouring of graphs}, Lecture notes of SGT
  2013, Oleron, France.

\bibitem{JLR00}
S.~Janson, T.~\L uczak, and A.~Rucinski, \emph{Random graphs},
  Wiley-Interscience Series in Discrete Mathematics and Optimization,
  Wiley-Interscience, New York, 2000.

\bibitem{KMO11}
D.~K\"{u}hn, R.~Mycroft, and D.~Osthus, \emph{A proof of {S}umner's universal
  tournament conjecture for large tournaments}, Proc. Lond. Math. Soc. (3)
  \textbf{102} (2011), 731--766.

\bibitem{NY19}
S.~Nassar and R.~Yuster, \emph{Acyclic subgraphs with high chromatic number},
  European J. Combin. \textbf{75} (2019), 11--18.

\bibitem{RW89}
V.~R\"{o}dl and P.~Winkler, \emph{A {R}amsey-type theorem for orderings of a
  graph}, SIAM J. Discrete Math. \textbf{2} (1989), 402--406.

\bibitem{Roy67}
B.~Roy, \emph{Nombre chromatique et plus longs chemins d'un graphe}, Rev.
  Fran\c{c}aise Informat. Recherche Op\'{e}rationnelle \textbf{1} (1967),
 129--132.

\bibitem{Vit62}
L.~M. Vitaver, \emph{Determination of minimal coloring of vertices of a graph
  by means of {B}oolean powers of the incidence matrix}, Dokl. Akad. Nauk SSSR
  \textbf{147} (1962), 758--759.

\end{thebibliography}

\end{document}